\documentclass[11pt]{article}

\usepackage[margin=1in]{geometry}
\usepackage[T1]{fontenc}
\usepackage{lmodern}
\usepackage{microtype}
\usepackage{amsmath,amssymb,amsthm,mathtools}
\usepackage{xcolor}
\usepackage[colorlinks=true,linkcolor=blue!55!black,citecolor=blue!55!black,urlcolor=blue!55!black]{hyperref}
\usepackage[capitalise,noabbrev]{cleveref}

\numberwithin{equation}{section}
\newtheorem{theorem}{Theorem}[section]
\newtheorem{proposition}[theorem]{Proposition}
\newtheorem{lemma}[theorem]{Lemma}
\newtheorem{corollary}[theorem]{Corollary}
\theoremstyle{definition}
\newtheorem{definition}[theorem]{Definition}
\newtheorem{remark}[theorem]{Remark}

\newcommand{\R}{\mathbb R}

\newcommand{\cA}{\mathcal A}
\newcommand{\cH}{\mathcal H}

\newcommand{\Tr}{\operatorname{Tr}}
\newcommand{\vol}{\operatorname{vol}}
\newcommand{\supp}{\operatorname{supp}}
\newcommand{\ee}{\mathrm e}
\newcommand{\dd}{\,\mathrm d}
\newcommand{\abs}[1]{\lvert #1\rvert}
\newcommand{\norm}[1]{\lVert #1\rVert}

\title{Spherical density for real semisimple groups}
\author{Christopher Lutsko}
\date{}

\begin{document}
\maketitle

\begingroup
\renewcommand{\thefootnote}{}
\footnotetext{2020 Mathematics Subject Classification:
22E46, 11F72. Keywords: spherical density, weak injective radius.}
\endgroup

\begin{abstract}
Let $G$ be a connected semisimple real Lie group with finite center and let
$(\Gamma_N)$ be a tower of cocompact lattices.  We prove that the weak
injective radius property implies the spherical
density hypothesis with the same parameter.  Together with the converse
theorem of Golubev and Kamber, this gives an equivalence in the cocompact
Archimedean setting.  The proof uses a uniform spherical-mass estimate near
spectral walls and a positive pre-trace kernel, but no tempered-spectrum
estimate or a priori spectral gap.
\end{abstract}

\section{Introduction}

Let $G$ be a connected semisimple real Lie group with finite center and let
$K$ be a maximal compact subgroup.  Let $X_N=\Gamma_N\backslash G/K$ be
finite covers of a fixed compact locally symmetric space; thus
$\Gamma_N<\Gamma_1$ has finite index, with no normality assumption.  Put
$V_N=[\Gamma_1:\Gamma_N]$, and let $m(\pi,\Gamma_N)$ be the multiplicity of
$\pi$ in $L^2(\Gamma_N\backslash G)$.  If $\pi$ is an irreducible unitary
representation of $G$, let $p(\pi)$ be the infimum of the $p\geq2$ for which
the $K$-finite matrix coefficients of $\pi$ belong to $L^{p+\epsilon}(G)$
for every $\epsilon>0$.

Let $\widehat G_{\mathrm{sph}}$ be the set of spherical irreducible unitary
representations.  Fix a positive elliptic Casimir operator $\Omega$.  If
$\chi_\pi(\Omega)$ is the scalar by which it acts on $\pi$, set
$\Lambda(\pi)=1+\abs{\chi_\pi(\Omega)}$.  For $L\geq0$ and $p\geq2$, put
\begin{equation}\label{eq:M-definition}
 M_{\mathrm{sph}}(L,\Gamma_N,p)
 =\sum_{\substack{\pi\in\widehat G_{\mathrm{sph}}\\
                   \Lambda(\pi)\leq L,\ p(\pi)\geq p}}
       m(\pi,\Gamma_N).
\end{equation}

Using the Cartan decomposition, we write $g=k\exp(H)k'$ with $H$ in the closed positive Cartan chamber and set
$\ell(g)=2\rho(H)$, where $\rho$ is the half-sum of the positive restricted
roots, counted with multiplicity.  Golubev and Kamber introduced a weak
injective radius property, recalled in \cref{def:wirp}, and formulated their
spherical density hypothesis as
an upper bound for $M_{\mathrm{sph}}(L,\Gamma_N,p)$ with the Sarnak--Xue
exponent.  They conjectured that the former implies the latter
\cite[Conjecture 1.7]{GK}, and proved this over non-Archimedean fields and in
real rank one.  In higher real rank the missing input was a spherical lower
bound uniform near Weyl walls; see \cite[Conjecture 3.12]{GK}.

Our main theorem proves the Archimedean spherical implication.

\begin{theorem}
\label{thm:main}
Let $G$ be a connected semisimple real Lie group with finite center, and let
$(\Gamma_N)$ be a tower of cocompact lattices satisfying the weak injective
radius property with parameter $0<\alpha\leq1$.  There is a constant
$C=C(G)>0$ such that, for every $L\geq0$, $p\geq2$, and $\epsilon>0$,
\begin{equation}\label{eq:main}
 M_{\mathrm{sph}}(L,\Gamma_N,p)
 \ll_{G,\Gamma_1,\alpha,\epsilon}(1+L)^C
 V_N^{1-\alpha(1-2/p)+\epsilon}.
\end{equation}
\end{theorem}

Golubev and Kamber proved the converse implication \cite[Theorem 1.6]{GK}.
Their argument preserves the parameter $\alpha$, and hence we obtain the
following.

\begin{corollary}\label{cor:equivalence}
For towers of cocompact lattices in $G$, the weak injective radius property
with parameter $\alpha$ is equivalent to the spherical density hypothesis
with parameter $\alpha$.
\end{corollary}

\subsection*{Relation to previous work}

Sarnak and Xue introduced density bounds as a quantitative substitute for
temperedness \cite{SX}.  Golubev and Kamber placed these estimates in a
geometric framework and isolated spherical density and weak injective radius
as two sides of the same conjectural picture \cite{GK}.  

There has also been substantial arithmetic progress.  Blomer proved strong
density estimates for $\mathrm{GL}_n$ \cite{Blomer}; Assing and Blomer treated
principal congruence subgroups \cite{AB}; and Assing--Blomer--Nelson obtained
arbitrary level and transferred the result to certain compact inner forms
\cite{ABN}.  Further results include horizontal families of arithmetic
lattices \cite{FHMM}, paramodular quotients \cite{AssingSp4}, and irreducible
lattices in products of rank-one groups \cite{Kelmer}.  These arguments use
arithmetic trace formulae, transfer, or the structure of the automorphic
spectrum.  The theorem here is an implication from a geometric counting
property and applies to an arbitrary cocompact tower.

The positive-kernel construction has a more direct lineage.  It refines Oh's
pre-trace argument for the exceptional spectrum of congruence covers
\cite[Section~2.3]{OhDensity}.  In the author's earlier density theorem for
thin groups \cite{LutskoDensity}, the same argument was sharpened by stopping
the kernel before the first nonidentity return.  The new feature here is the
localization of the kernel in a critical Cartan direction, together with
uniform control as the spectral parameter approaches a wall.

The representation-theoretic inputs are different as well.  The polyhedral
description of $p(\pi)$ is the same geometry which appears in \cite{LWW}.
The expansion at singular spectral parameters is due to Narayanan, Pasquale,
and Pusti \cite{NPP}.  Their construction is local in the spectral parameter.
The required global uniformity is obtained below from their explicit
regularization and a finite induction over affine root hyperplanes.

\subsection*{Proof mechanism}

The invariant $p(\pi)$ measures the decay of the matrix coefficients of
$\pi$; larger values correspond to slower decay.  Let $\varphi_\lambda$ be
the normalized spherical function and put
\[
 \vartheta(\lambda)=1-\frac2{p(\pi_\lambda)}
 \qquad\text{and}\qquad Q_\lambda=2+\norm{\lambda}.
\]
By \cref{lem:p-poly}, $\vartheta(\lambda)$ is the maximum over the finitely
many normalized extreme rays of the positive Cartan chamber.  The rays and the polyhedral
formula are explained in \cref{sec:spectral-parameters}; this is the same
description used in \cite{LWW}.  Choose a direction where the maximum is
attained and move it a logarithmic distance into the interior of the chamber.
The move costs only a power of $Q_\lambda R$, and the Harish--Chandra
expansion then gives a uniform amount of spherical mass.  At a singular
parameter we use the all-parameter expansion of Narayanan, Pasquale, and
Pusti \cite{NPP}; a finite induction over the singular hyperplanes prevents
cancellation of the leading term.  We obtain a cutoff $\chi_{\lambda,R}$
supported where $\ell(g)=R+O(1)$ and
\[
 \int_G\chi_{\lambda,R}(g)\ee^{\ell(g)/p(\pi_\lambda)}
       \abs{\varphi_\lambda(g)}^2\dd g
 \gg Q_\lambda^{-C}R^{-C}
       \ee^{R(1-1/p(\pi_\lambda))}.
\]

Define
\[
 \beta_{\lambda,R}(g)
 =\chi_{\lambda,R}(g)\ee^{\ell(g)/p(\pi_\lambda)}
       \overline{\varphi_\lambda(g^{-1})},
 \qquad
 \Phi_{\lambda,R}=\beta_{\lambda,R}^* *\beta_{\lambda,R}.
\]
Then $\Phi_{\lambda,R}$ is positive on the unitary dual,
\[
\widehat\Phi_{\lambda,R}(\lambda)
 \gg Q_\lambda^{-C}R^{-C}
       \ee^{2R(1-1/p(\pi_\lambda))},
\]
while the convolution estimate of \cite{GK} gives
\[
 \abs{\Phi_{\lambda,R}(g)}
 \ll Q_\lambda^CR^C\ee^{R-\ell(g)/2}
      \mathbf1_{\ell(g)\leq2R+O(1)}.
\]
This is the pre-trace construction of \cite[Section~2.3]{OhDensity}, with
the kernel localized in the critical Cartan direction.  A polynomially sized
net of these kernels amplifies every parameter being counted.  The weak
injective radius property bounds the geometric trace by
$V_N^{1+\epsilon}\ee^R$, up to polynomial factors, whereas each counted
representation contributes at least $\ee^{2R(1-1/p)}$.  Taking
$R=\alpha\log V_N+O(1)$ gives \eqref{eq:main}.

\section{Preliminaries}

We recall the notation from spherical harmonic analysis that is used in the
proof.  Compact factors of $G$ may be included in $K$, so we assume that
$G$ has no compact factors.  Let $\mathfrak g$ and $\mathfrak k$ be the Lie
algebras of $G$ and $K$.  The Cartan decomposition of the Lie algebra is
\[
                         \mathfrak g=\mathfrak k\oplus\mathfrak p.
\]
Fix a maximal abelian subspace $\mathfrak a\subset\mathfrak p$.  Let $\Sigma$
be the restricted root system, with root multiplicities $m_\alpha$, and
choose positive roots $\Sigma^+$ with simple roots $\Delta$.  The open
positive Cartan chamber is
\[
 \mathfrak a^+=\{H\in\mathfrak a:\alpha(H)>0
                         \text{ for every }\alpha\in\Delta\}.
\]
We write $\overline{\mathfrak a^+}$ for its closure and put
\[
 \rho=\frac12\sum_{\alpha\in\Sigma^+}m_\alpha\alpha.
\]

\subsection*{Notation}

The normalizer $N_K(\mathfrak a)$ consists of the elements of $K$ which
preserve $\mathfrak a$, and the centralizer $Z_K(\mathfrak a)$ consists of
those which act trivially on it.  Their quotient
\[
                         W=N_K(\mathfrak a)/Z_K(\mathfrak a)
\]
is the \emph{restricted Weyl group}.  Fix a $W$-invariant inner product
$\langle\cdot,\cdot\rangle$ on $\mathfrak a$.  For
$\alpha\in\mathfrak a^*$, let $\alpha^\sharp\in\mathfrak a$ be defined by
$\alpha(H)=\langle\alpha^\sharp,H\rangle$.  The \emph{coroot} of $\alpha$
is
\[
 \alpha^\vee=
 \frac{2\alpha^\sharp}{\langle\alpha^\sharp,\alpha^\sharp\rangle};
 \qquad \alpha(\alpha^\vee)=2.
\]
Thus $\lambda(\alpha^\vee)$ means simply that the functional $\lambda$ is
evaluated at the vector $\alpha^\vee$.  The inner product identifies
$\mathfrak a$ with $\mathfrak a^*$ and gives norms on these spaces and on
their complexifications.  Write
$\mathfrak a_{\mathbb C}^*=\mathfrak a^*\otimes_\R\mathbb C$.  On
$\mathfrak a_{\mathbb C}^*/W$ we use the quotient distance
\[
 d_W(\lambda,\nu)=\min_{w\in W}\norm{\lambda-w\nu}.
\]
A functional $\eta\in\mathfrak a^*$ is
\emph{dominant} if
\[
              \eta(\alpha^\vee)\geq0
              \qquad(\alpha\in\Delta),
\]
and every $W$-orbit in $\mathfrak a^*$ has a dominant representative.  For
$X,Y\in\overline{\mathfrak a^+}$, we write $X\preceq Y$ if
\[
                      Y-X=\sum_{\alpha\in\Delta}t_\alpha\alpha^\vee
                      \quad\text{with }t_\alpha\geq0.
\]
This is the dominance order on the closed chamber.

We write $A\ll_S B$ if $\abs A\leq C_S B$, and $A\asymp_S B$ if both
$A\ll_S B$ and $B\ll_S A$.  The indicator of a set $E$ is denoted by
$\mathbf1_E$.  The space $C_c^\infty(G)$ consists of the smooth,
compactly supported functions on $G$, and
$C_c^\infty(K\backslash G/K)$ is its subspace of functions satisfying
$f(k_1gk_2)=f(g)$ for $k_1,k_2\in K$.  We call these functions
$K$-bi-invariant.  We normalize Haar measure on $K$ to have mass one.  Our
convolution and involution conventions are
\[
 (f_1*f_2)(g)=\int_G f_1(h)f_2(h^{-1}g)\dd h,
 \qquad f^*(g)=\overline{f(g^{-1})}.
\]

The Cartan decomposition says that every $g\in G$ can be written as
$g=k\exp(H)k'$ with $k,k'\in K$ and a unique
$H\in\overline{\mathfrak a^+}$.  This $H$ is the \emph{Cartan projection}
$\mu(g)$.  We use the scalar length
\[
                         \ell(g)=2\rho(\mu(g)).
\]
The vector-valued triangle inequality for this order,
$\mu(gh)\preceq\mu(g)+\mu(h)$, is proved, for example, in
\cite[Remark~3.36(ii)]{KLM}.  Since $\rho$ is dominant, it gives
\begin{equation}\label{eq:length-properties}
 \ell(g^{-1})=\ell(g),\qquad \ell(gh)\leq\ell(g)+\ell(h).
\end{equation}
In particular, $\ell$ is a proper, continuous, $K$-bi-invariant length.

\subsection{Spherical parameters and the \texorpdfstring{$p$}{p}-invariant}
\label{sec:spectral-parameters}

For $\pi\in\widehat G_{\mathrm{sph}}$, choose a unit vector $v\in\pi^K$ and
write
\[
                    \varphi_\pi(g)=\langle\pi(g)v,v\rangle,
                    \qquad \varphi_\pi(e)=1.
\]
Normalized spherical functions are parameterized by
$\lambda\in\mathfrak a_{\mathbb C}^*/W$, where
the quotient identifies $\lambda$ with $w\lambda$ for $w\in W$.  We write
$\varphi_\lambda$ and $\pi_\lambda$ for the corresponding function and
representation.  We normalize the parameter so that the exponents in the
Harish--Chandra expansion are $w\lambda-\rho$; in the convention which writes
$iw\lambda-\rho$, our $\lambda$ is replaced by $i\lambda$.  In particular,
$\varphi_\lambda=\varphi_{w\lambda}$.  When useful, we choose a
representative for which $\operatorname{Re}\lambda$ is dominant.  By the
Helgason--Johnson theorem, a bounded spherical function has a representative
with
\[
                 \operatorname{Re}\lambda\in\operatorname{conv}(W\rho).
\]
Here $\operatorname{conv}(W\rho)$ is the convex hull of the finite Weyl
orbit of $\rho$; see \cite[Theorem~IV.8.1]{Helgason}.  We use the spectral
size
\[
                              Q_\lambda=2+\norm\lambda.
\]
For the positive elliptic Casimir operator $\Omega$ fixed in the
introduction, let $\chi_\lambda(\Omega)$ be its scalar action on
$\pi_\lambda$.  Thus
\[
                 \Lambda(\pi_\lambda)=1+\abs{\chi_\lambda(\Omega)}.
\]
Since the real part of a unitary spherical parameter lies in a fixed compact
set, the Harish--Chandra isomorphism
\cite[Theorem~II.5.18]{Helgason} gives
\begin{equation}\label{eq:height-comparison}
 Q_\lambda^2\asymp_G 1+\Lambda(\pi_\lambda)
 \qquad(\pi_\lambda\in\widehat G_{\mathrm{sph}}).
\end{equation}

The closed chamber $\overline{\mathfrak a^+}$ is a cone cut out by the
linear inequalities $\alpha(H)\geq0$, $\alpha\in\Delta$.  A ray
$\R_{\geq0}H$ in this cone is \emph{extreme} if $H=H'+H''$ with
$H',H''\in\overline{\mathfrak a^+}$ forces $H'$ and $H''$ to lie on the
same ray.  Equivalently, the extreme rays are the one-dimensional faces of
the cone.  Choose one vector $H_i$ on each extreme ray and normalize it by
$2\rho(H_i)=1$.  If there are $r=\dim\mathfrak a$ such rays, define
\begin{equation}\label{eq:theta-poly}
 \vartheta(\lambda)
 =\max_{w\in W}\max_{1\leq i\leq r}
       \frac{\operatorname{Re}(w\lambda)(H_i)}{\rho(H_i)}.
\end{equation}

\begin{lemma}\label{lem:p-poly}
For a spherical unitary representation $\pi_\lambda$,
\[
              \vartheta(\lambda)=1-\frac2{p(\pi_\lambda)}.
\]
Here $p(\pi_\lambda)=\infty$ when $\vartheta(\lambda)=1$.
\end{lemma}

\begin{proof}
The leading-exponent criterion for matrix coefficients
\cite[Theorem~9.2]{GK} identifies the right-hand side with the maximum of
$\max_w\operatorname{Re}(w\lambda)(H)/\rho(H)$ over nonzero $H$ in the
positive chamber.  This is the polyhedral norm used in
\cite[Proposition~3.3]{LWW}.  Because the quotient is unchanged when $H$
is multiplied by a positive scalar, it is enough to consider the compact
simplex $2\rho(H)=1$.  On that simplex the function is convex and
piecewise linear.  It is therefore maximized at a vertex, and these vertices
are precisely the normalized extreme rays $H_i$.
\end{proof}

For $f\in C_c^\infty(K\backslash G/K)$, its \emph{spherical transform} is
\[
 \widehat f(\lambda)=\int_G f(g)\varphi_\lambda(g^{-1})\dd g.
\]
The integrated operator $\pi_\lambda(f)=\int_Gf(g)\pi_\lambda(g)\dd g$
acts on the line $\pi_\lambda^K$ by the scalar
$\widehat f(\lambda)$.
For an arbitrary unitary representation $\pi$, the same notation means
$\pi(f)=\int_G f(g)\pi(g)\dd g$.
Thus $\widehat{f^* *f}(\lambda)=\abs{\widehat f(\lambda)}^2$ on the spherical
unitary dual.  Moreover, if $f$ is $K$-bi-invariant, then
$\operatorname{Tr}\pi_\lambda(f)=\widehat f(\lambda)$ for a spherical
representation, while $\pi(f)=0$ for a nonspherical representation.

\subsection{The two density properties}

Normalize the Haar measure by $\vol(\Gamma_1\backslash G)=1$, so that
$\vol(\Gamma_N\backslash G)=V_N$.

\begin{definition}\label{def:wirp}
The tower $(\Gamma_N)$ has the weak injective radius property with parameter
$0<\alpha\leq1$ if there exists $C_0>0$ such that, uniformly for
$0\leq d\leq2\alpha\log V_N-C_0$,
\begin{equation}\label{eq:wirp}
 \frac1{V_N}\int_{\Gamma_N\backslash G}
 \#\{\gamma\in\Gamma_N:\ell(x^{-1}\gamma x)\leq d\}\dd x
 \ll_\epsilon V_N^\epsilon\ee^{d/2}.
\end{equation}
For each $\epsilon>0$, the implied constant is independent of $N$ and $d$.
\end{definition}

\begin{remark}
\label{rem:wirp-formulations}
For a cocompact base lattice, \cref{def:wirp} is equivalent to the weak
injective radius property of \cite[Section 2.2]{GK}, with the same parameter
$\alpha$.  Indeed, choose a compact fundamental set $\Omega\subset G$ for
$\Gamma_1\backslash G$ and decompose $\Gamma_N\backslash G$ into its
$V_N$ translates.  The two averages differ only by conjugation by an element
$y\in\Omega$, and
\[
 \abs{\ell(y^{-1}gy)-\ell(g)}\leq2\ell(y)\leq C_\Omega.
\]
Thus each count is bounded by the other with the radius changed by at most
$C_\Omega$.  This preserves both the estimate and the parameter $\alpha$.
\end{remark}

\begin{remark}
The parameter $\alpha$ can be estimated directly in arithmetic examples.
Golubev and Kamber do this for principal congruence towers and, using general
results on congruence subgroups, obtain an explicit positive parameter for a
broader class of congruence towers; see \cite[Corollary~2.2 and
Theorem~2.3]{GK}.  Thus \cref{thm:main} gives unconditional spherical density
bounds for the cocompact congruence towers covered by their results.
\end{remark}

\begin{definition}
With $M_{\mathrm{sph}}$ as in \eqref{eq:M-definition}, the tower has the
spherical density hypothesis with parameter $\alpha$ if there is $C=C(G)>0$
such that
\[
 M_{\mathrm{sph}}(L,\Gamma_N,p)
 \ll_\epsilon(1+L)^C
 V_N^{1-\alpha(1-2/p)+\epsilon}
\]
for all $L\geq0$, $p\geq2$, and $\epsilon>0$.
\end{definition}

The exponent $C$ is independent of $p$ and $\epsilon$.  This uniformity in
the spectral aspect is used in \cref{sec:density-proof}.

\subsection{The geometric trace bound}

For $g\in G$, the right-regular action on $L^2(\Gamma_N\backslash G)$ is
$R_{\Gamma_N}(g)u(x)=u(xg)$.  For $F\in C_c^\infty(G)$, its integrated
operator is
\[
 (R_{\Gamma_N}(F)u)(x)=\int_G F(g)u(xg)\dd g.
\]
Its automorphic kernel is
$K_F(x,y)=\sum_{\gamma\in\Gamma_N}F(x^{-1}\gamma y)$.
Because $\Gamma_N\backslash G$ is compact, the right-regular representation
has a discrete spectral decomposition.  For smooth compactly supported $F$,
\begin{equation}\label{eq:spectral-trace}
 \Tr R_{\Gamma_N}(F)
 =\sum_{\pi}m(\pi,\Gamma_N)\Tr\pi(F),
\end{equation}
where the sum is over irreducible unitary representations occurring in
$L^2(\Gamma_N\backslash G)$.  If $F$ is $K$-bi-invariant, only spherical
representations contribute.

\begin{lemma}\label{lem:geometric-trace}
There is $C_1=C_1(G,C_0)$ with the following property.  Suppose
$R\leq\alpha\log V_N-C_1$ and $F=F^*\in C_c^\infty(G)$ satisfies
\begin{equation}\label{eq:F-majorant}
 \abs{F(g)}\leq A_R\ee^{R-\ell(g)/2}
       \mathbf1_{\ell(g)\leq2R+C_1}.
\end{equation}
Then
\[
 \abs{\Tr R_{\Gamma_N}(F)}
 \ll_{G,\epsilon} A_R(1+R)V_N^{1+\epsilon}\ee^R.
\]
\end{lemma}

\begin{proof}
This is the shell argument used in
\cite[Section~2.3]{OhDensity}; compare also
\cite[proof of Proposition~3.9 and Corollary~5.4]{GK}.  Integrating the
automorphic kernel on the diagonal gives
\[
 \abs{\Tr R_{\Gamma_N}(F)}
 \leq\int_{\Gamma_N\backslash G}
       \sum_{\gamma\in\Gamma_N}
       \abs{F(x^{-1}\gamma x)}\dd x.
\]
Take $C_1\geq C_0+2$ and split the support into the shells
$m\leq\ell<m+1$, $m\in\mathbb Z_{\geq0}$.  Then every radius $m+1$ below
is in the range of \eqref{eq:wirp}.  On the $m$th shell,
\eqref{eq:F-majorant} and \eqref{eq:wirp} give
\begin{align*}
 \abs{\Tr R_{\Gamma_N}(F)}
 &\leq A_R\ee^R
   \sum_{\substack{m\in\mathbb Z_{\geq0}\\m\leq2R+C_1}}\ee^{-m/2}
   \int_{\Gamma_N\backslash G}
   \#\{\gamma\in\Gamma_N:\ell(x^{-1}\gamma x)<m+1\}\dd x\\
 &\ll_{G,\epsilon}A_R\ee^R V_N^{1+\epsilon}
   \sum_{\substack{m\in\mathbb Z_{\geq0}\\m\leq2R+C_1}}
       \ee^{-m/2}\ee^{(m+1)/2}\\
 &\ll_{G,\epsilon}A_R(1+R)V_N^{1+\epsilon}\ee^R.
\end{align*}
\end{proof}

We also use the local Weyl bound
\begin{equation}\label{eq:local-weyl}
 M_{\mathrm{sph}}(L,\Gamma_N,2)
 \ll_{G,\Gamma_1}V_N(1+L)^d
\end{equation}
for some $d=d(G)$.  This follows from the heat-kernel proof of Weyl's law;
see \cite{Hormander}.  The diagonal heat kernel on each cover is locally
the lift of the heat-kernel parametrix on the fixed base.  Its diagonal
upper bound is therefore independent of the cover, and integration over
$\Gamma_N\backslash G/K$ introduces the factor $V_N$.

\section{A uniform spherical lower bound}

This section proves the analytic estimate used by the positive kernel.  We
need a lower bound for $\varphi_\lambda$ which is uniform in both the Cartan
variable $H$ and the spectral parameter $\lambda$.  There are two possible
losses of uniformity.  First, the direction which controls $p(\pi_\lambda)$
may lie on the boundary of the positive Cartan chamber.  Second, the usual
Harish--Chandra formula has singular-looking coefficients when $\lambda$
approaches certain hyperplanes in the spectral parameter space.  We deal
with the first problem by moving the Cartan direction a distance
$O(\log(Q_\lambda R))$ into the chamber.  We deal with the second by using
the expansion of Narayanan, Pasquale, and Pusti at singular parameters and
then comparing nearby parameters.

\subsection{The critical direction}

Choose $H_\circ\in\mathfrak a^+$ with $2\rho(H_\circ)=1$.  This is a fixed
direction in the interior of the chamber.  By Weyl invariance, we choose the
representative of $\lambda$ for which $\operatorname{Re}\lambda$ is
dominant.  We call an index $i$ \emph{critical} if $H_i$ attains the maximum
in \eqref{eq:theta-poly}.  For such an index,
\[
             \operatorname{Re}\lambda(H_i)
             =\frac12\vartheta(\lambda).
\]
Let $a_0=a_0(G)$ be a sufficiently large constant, to be fixed in the proof
of \cref{lem:wall-induction}, and put
\begin{equation}\label{eq:log-push}
 \tau_{\lambda,R}=a_0\log(2+Q_\lambda R),\qquad
 H_{i,\lambda,R}(u)
 =(u-\tau_{\lambda,R})H_i+\tau_{\lambda,R}H_\circ.
\end{equation}
The variable $u$ is the radial variable for our length: since
$2\rho(H_i)=2\rho(H_\circ)=1$, we have
\[
                 2\rho(H_{i,\lambda,R}(u))=u.
\]
The point $uH_i$ may lie on a wall.  Formula \eqref{eq:log-push} moves it
by $\tau_{\lambda,R}(H_\circ-H_i)$ while keeping the same value of $u$.
If
$u\geq2\tau_{\lambda,R}$, every positive root satisfies
\begin{equation}\label{eq:root-distance}
 \alpha(H_{i,\lambda,R}(u))\gg_G\log(2+Q_\lambda R),
\end{equation}
and
\begin{equation}\label{eq:critical-loss}
 \operatorname{Re}\lambda(H_{i,\lambda,R}(u))
 \geq\frac12\vartheta(\lambda)u-C_G\log(2+Q_\lambda R).
\end{equation}
Indeed,
\[
 \alpha(H_{i,\lambda,R}(u))
 =(u-\tau_{\lambda,R})\alpha(H_i)
   +\tau_{\lambda,R}\alpha(H_\circ)
 \geq\tau_{\lambda,R}\alpha(H_\circ),
\]
which proves \eqref{eq:root-distance}.  Since $i$ is critical,
\begin{align*}
 \operatorname{Re}\lambda(H_{i,\lambda,R}(u))
 &=(u-\tau_{\lambda,R})\frac{\vartheta(\lambda)}2
   +\tau_{\lambda,R}\operatorname{Re}\lambda(H_\circ)\\
 &\geq\frac12\vartheta(\lambda)u-C_G\tau_{\lambda,R}.
\end{align*}
Here we used
$\operatorname{Re}\lambda\in\operatorname{conv}(W\rho)$, so
$\abs{\operatorname{Re}\lambda(H_\circ)}\ll_G1$.  This proves
\eqref{eq:critical-loss}.

Let
\[
             \ker\rho=\{Z\in\mathfrak a:\rho(Z)=0\}.
\]
Fix bounded open rectangular boxes, centered at the origin,
\[
             B_0\Subset B_1\Subset B_2\subset\ker\rho,
\]
where $B_j\Subset B_{j+1}$ means that the closure of $B_j$ is contained in
the interior of $B_{j+1}$.  The variable $Z\in\ker\rho$ moves within a level
set of the length.  Choose $A_0=A_0(G)$ sufficiently large compared with
$a_0$ and suppose that
\begin{equation}\label{eq:long-radius}
 R\geq A_0\log(2+Q_\lambda R),
\end{equation}
By \eqref{eq:root-distance}, the following sets lie in the open chamber
after $A_0$ is enlarged if necessary.  Choose a smooth function supported
in the open chamber which is one on
\[
 \{H_{i,\lambda,R}(u)+Z:R-2\leq u\leq R-1,\ Z\in B_1\}
\]
and is supported in
\[
 \{H_{i,\lambda,R}(u)+Z:R-3\leq u\leq R,\ Z\in B_2\}.
\]
Extend this function to a smooth $W$-invariant function on $\mathfrak a$ by
Weyl reflection and by zero outside the reflected supports.  Its composition
with the Cartan projection is then a smooth, nonnegative, $K$-bi-invariant
cutoff $\chi_{\lambda,i,R}$ on $G$.  Its support satisfies
\[
 \supp\chi_{\lambda,i,R}\subset
 \{g:R-3\leq\ell(g)\leq R\}.
\]
In Cartan coordinates, Haar measure has the form
\[
 \int_G f(g)\dd g
 =c_G\int_K\int_{\mathfrak a^+}\int_K
 f(k\exp(H)k')J(H)\dd k\dd H\dd k',
 \qquad
 J(H)=\prod_{\alpha\in\Sigma^+}(\sinh\alpha(H))^{m_\alpha}.
\]
The positive constant $c_G$ depends on the normalization of Haar measure.
Every positive root is large throughout the support of the cutoff.  Hence
\begin{equation}\label{eq:jacobian-tube}
 J(H_{i,\lambda,R}(u)+Z)\asymp_G\ee^u
\end{equation}
there, because $2\rho(H_{i,\lambda,R}(u)+Z)=u$ and
$\sinh\alpha(H)\asymp\ee^{\alpha(H)}$ for $\alpha\in\Sigma^+$.

\subsection{A finite exponential estimate}

We first record the only finite-dimensional estimate used below.  The
important point is that the number of exponents and the polynomial degrees
are fixed.

\begin{lemma}
\label{lem:finite-exponential}
Let $B\subset\R^n$ be a box with nonempty interior, and fix
$M,D,C_0\geq1$.  Suppose that the $P_j$ are complex polynomials,
$\zeta_j\in\mathbb C^n$, and
\[
 E(x)=\sum_{j=1}^mP_j(x)\ee^{\zeta_j\cdot x},
 \qquad m\leq M,\quad\deg P_j\leq D,
\]
where $\norm{\operatorname{Re}\zeta_j}\leq C_0$,
$\norm{\zeta_j}\leq Q$ with $Q\geq2$, and
$\norm{\zeta_j-\zeta_k}\geq\delta$ for $j\neq k$.  If
$\mathcal C(E)$ is the largest absolute value of a coefficient of one of the
$P_j$, then
\[
 \norm{E}_{L^2(B)}
 \geq c\,Q^{-C}\delta^C\mathcal C(E),
\]
where $c,C$ depend only on $B,M,D,C_0$.
\end{lemma}

This is the fixed-complexity form of the standard quantitative independence
estimate for quasipolynomials.   A reference for
the multivariable Bernstein estimate is \cite{BrudnyiBernstein}; the
coefficient estimate, including polynomial factors, follows from the bound
for inverse confluent Vandermonde matrices in
\cite[Theorem~1]{BatenkovVandermonde}.

\subsection{The all-parameter expansion}

We first describe the hyperplanes which occur in the spectral parameter
space.  Use the inner product to put
\[
 \lambda_\alpha=
 \frac{\langle\lambda,\alpha\rangle}{\langle\alpha,\alpha\rangle}
 =\frac12\lambda(\alpha^\vee),
 \qquad \alpha\in\Sigma.
\]
For $n\in\mathbb Z$, the set
\begin{equation}\label{eq:affine-root-hyperplane}
 \mathcal H_{\alpha,n}
 =\{\lambda\in\mathfrak a_{\mathbb C}^*:\lambda_\alpha=n\}
\end{equation}
is an \emph{affine root hyperplane}.  It is called affine because it is a
translate of the linear root hyperplane $\lambda_\alpha=0$.  These
hyperplanes lie in the spectral parameter space; they should not be confused
with the walls $\alpha(H)=0$ of the Cartan chamber.

If $\lambda$ lies on none of these hyperplanes, the Harish--Chandra expansion
has the form
\begin{equation}\label{eq:regular-HC-schematic}
 \varphi_\lambda(\exp H)
 =\sum_{w\in W}c(w\lambda)\ee^{(w\lambda-\rho)(H)}
   \sum_{\mu\in\Lambda^+}\Gamma_\mu(w\lambda)\ee^{-\mu(H)}.
\end{equation}
Here $c$ is the Harish--Chandra $c$-function,
\[
 \Lambda^+=\left\{\sum_{\alpha\in\Delta}n_\alpha\alpha:
                         n_\alpha\in\mathbb Z_{\geq0}\right\}
\]
is the positive root lattice, and
$\operatorname{ht}(\mu)=\sum_{\alpha\in\Delta}n_\alpha$ is the height of
$\mu$.  The coefficients $c(w\lambda)$ and
$\Gamma_\mu(w\lambda)$ may have poles on hyperplanes of the form
\eqref{eq:affine-root-hyperplane}, although their Weyl sum, which is the
spherical function, is regular.  When exponents meet on such a hyperplane,
regularization replaces some exponentials by a polynomial in $H$ times an
exponential.  The elementary limit
\[
             \frac{\ee^{zt}-\ee^{-zt}}{z}\longrightarrow 2t
             \qquad(z\longrightarrow0)
\]
is the rank-one model for this phenomenon.

Fix a compact, convex, $W$-invariant neighborhood
$U_0\subset\mathfrak a^*$ of
$\operatorname{conv}(W\rho)$ and put
\[
                     \mathcal U=U_0+i\mathfrak a^*.
\]
The singular hyperplanes of the coefficients in
\eqref{eq:regular-HC-schematic} have integer translates as in
\eqref{eq:affine-root-hyperplane}.  Since $U_0$ is compact, only finitely
many of them meet $\mathcal U$.  Enlarge this finite collection to be
$W$-invariant and denote it by $\cH$.  A \emph{stratum} of $\cH$ is a connected
component of the set of parameters which lie on one fixed collection of
hyperplanes and on no others.  A \emph{proper substratum} is obtained by
imposing at least one additional hyperplane equation.  Distances are measured
using the norm fixed in the notation subsection.

\begin{proposition}
\label{prop:uniform-HC}
There are $C,d\geq1$ with the following property.  Let $S$ be a stratum,
$\lambda\in S$, suppose that $\operatorname{Re}\lambda$ is dominant, and
suppose, for some $0<\delta\leq1$, that $\lambda$ has distance at least
$\delta$ from every proper
substratum of $S$.  For
$H\in\mathfrak a^+$ with $\min_{\alpha\in\Delta}\alpha(H)\geq1$,
\[
 \varphi_\lambda(\exp H)=\mathcal L_\lambda(H)+\mathcal R_\lambda(H),
\]
where the following hold.
\begin{enumerate}
\item The leading part $\mathcal L_\lambda$ consists of the terms with
      $\mu=0$.  After equal exponentials have been grouped, it is a sum of
      at most $d$ terms $P(H)\ee^{(w\lambda-\rho)(H)}$, with $\deg P\leq d$.
\item The coefficients of these polynomials, and their parameter derivatives
      of order at most $d$, are bounded by $(Q_\lambda/\delta)^C$.
\item Let $P_{\lambda,e}(H)$ be the polynomial multiplying
      $\ee^{(\lambda-\rho)(H)}$ after all equal exponentials have been
      grouped.  At least one coefficient of $P_{\lambda,e}$ has absolute
      value at least $(Q_\lambda/\delta)^{-C}$.
\item For every $A>0$, if
      $\min_{\alpha\in\Delta}\alpha(H)\geq
      a_A\log(2+Q_\lambda/\delta)$, then
\begin{equation}\label{eq:HC-remainder}
 \abs{\mathcal R_\lambda(H)}
 \ll_{G,A}(1+\norm H)^C (Q_\lambda/\delta)^{-A}
 \ee^{-\rho(H)+\max_w\operatorname{Re}(w\lambda)(H)}.
\end{equation}
\end{enumerate}
All constants are uniform over the finitely many strata.
\end{proposition}

\begin{proof}
The regularization at a singular parameter is
\cite[Proposition~2.5]{NPP}, the resulting all-parameter expansion is
\cite[Theorem~2.11]{NPP}, and the nonvanishing of the polynomial
$P_{\lambda,e}$ in part~(3) is \cite[Corollary~2.10]{NPP}.  We record why
the constants in those formulas grow at most polynomially in $Q_\lambda$
and $\delta^{-1}$.

The possible poles are simple and lie in $\cH$ by
\cite[Theorem~1.2 and Lemma~1.9]{NPP}.  On a fixed stratum, the regularizing
polynomials $p$ and $\pi$ of \cite[Proposition~2.5]{NPP} are products of the
affine linear forms defining the singular hyperplanes which contain that
stratum.  Multiplication by $p$ removes the poles along those hyperplanes
before the differential operator $\partial(\pi)$ is applied.  Every other
singular hyperplane is at distance at least $\delta$.  Consequently, the
terms with bounded $\operatorname{ht}(\mu)$, together with the finitely many
parameter derivatives used in the regularization, are bounded by a fixed
power of $Q_\lambda/\delta$.  The gamma-product formula for $c$ and Stirling's
formula give the same polynomial bound as the imaginary part of $\lambda$
grows.

For larger $m=\operatorname{ht}(\mu)$, the recurrence defining the
Harish--Chandra coefficient, recorded in
\cite[equation~(12) and Lemma~1.4]{NPP}, has denominator
\[
 \abs{\langle\mu,\mu-2\lambda\rangle}
 \geq\langle\mu,\mu-2\operatorname{Re}\lambda\rangle
 \gg_G(1+m)^2,
\]
once $m$ is larger than a fixed constant.  There are $O_G(1+m)$ terms in
the recurrence, and their numerators are $O_G(Q_\lambda+m)$.  If
$\Gamma_{\mu,S}$ denotes the coefficient after regularization on $S$,
induction on $m$, followed by induction on the number of parameter
derivatives, gives
\[
 \abs{\partial_\lambda^\kappa\Gamma_{\mu,S}(\lambda)}
 \leq\bigl(C(Q_\lambda/\delta)^C\bigr)^{1+m},
 \qquad \abs\kappa\leq d.
\]
Since
$\mu(H)\gg_Gm\min_{\alpha\in\Delta}\alpha(H)$, each increase in $m$
also supplies exponential decay in $H$.  The logarithmic lower bound in
part~(4) makes this decay larger than the coefficient growth.  The remaining
sum over $m$ is then a geometric series, which gives
\eqref{eq:HC-remainder}.

Finally, in the notation of \cite[Lemma~2.6]{NPP},
$b_0(\lambda)=\pi_0(\lambda)c(\lambda)$ is neither zero nor infinite on the
stratum.  Corollary~2.10 of that paper shows that the highest-degree
coefficient of $P_{\lambda,e}$ is a nonzero constant multiple of
$b_0(\lambda)$.  The product formula for $c$ in
\cite[equation~(19)]{NPP}, Stirling's formula, and the separation from the
remaining hyperplanes give the lower bound in part~(3).
\end{proof}

\subsection{Uniformity at spectral walls}

The previous proposition is uniform as long as the parameter is not too
close to a smaller stratum.  We now remove that condition.  Away from a
smaller stratum we apply the expansion directly.  Near a smaller stratum we
move the parameter onto it and use continuity.  Since there are only
finitely many hyperplanes, this process stops after finitely many steps.

\begin{lemma}[Finite wall induction]\label{lem:wall-induction}
Let $\lambda\in\mathcal U$, choose its Weyl representative so that
$\operatorname{Re}\lambda$ is dominant, and suppose
\eqref{eq:long-radius} holds.  Then, for every $i$,
\begin{equation}\label{eq:tube-L2-lower}
 \int_{R-2}^{R-1}\int_{B_0}
 \abs{\varphi_\lambda
  (\exp(H_{i,\lambda,R}(u)+Z))}^2\dd Z\dd u
 \gg_G Q_\lambda^{-C}R^{-C}
 \ee^{-R+2\operatorname{Re}\lambda
             (H_{i,\lambda,R}(R))}.
\end{equation}
\end{lemma}

\begin{proof}
Put $X=2+Q_\lambda R$ and
\[
 E_{\lambda,i,R}
 =\exp\left(-\frac R2+
       \operatorname{Re}\lambda(H_{i,\lambda,R}(R))\right).
\]
We prove by induction on the dimension of the stratum containing $\lambda$
that the square root of the left-hand side of
\eqref{eq:tube-L2-lower} is at least
$X^{-A_S}E_{\lambda,i,R}$ for some exponent $A_S$ depending only on the
stratum $S$.  Since $\cH$ is finite, the exponents may be replaced at the
end by one constant depending only on $G$.

During the induction we also allow the box $B_0$ to be translated by the
small displacements produced when the parameter is moved to a substratum.
An induction chain has length at most $\#\cH$, so these translations remain
in a fixed compact family of boxes contained in $B_1$; bounded values of $X$
are absorbed into the final constant.  The direct estimate below is uniform
over this family because \cref{lem:finite-exponential} is uniform over a
compact family of translated boxes.
We denote the accumulated translation by $Z_0$; at the start of the
induction $Z_0=0$.

We first give the direct argument.  Write $s=u-R$ and
$H_R=H_{i,\lambda,R}(R)$.  Then
\[
 H_{i,\lambda,R}(u)+Z=H_R+sH_i+Z,
 \qquad -2\leq s\leq-1.
\]
Suppose that $\lambda$ is at distance at least $\delta$ from every proper
substratum of $S$.  By \cref{prop:uniform-HC}, the leading part can be
written as
\[
 \mathcal L_\lambda(H_R+sH_i+Z)
 =\ee^{(\lambda-\rho)(H_R)}
   \sum_{j=1}^m \widetilde P_j(s,Z)
          \ee^{\zeta_j(s,Z)},
\]
where $m$ and the degrees of the polynomials $\widetilde P_j$ are bounded in
terms of $G$.  More explicitly, a term
$P(H)\ee^{(w\lambda-\rho)(H)}$ contributes
\[
 \widetilde P_j(s,Z)
 =P(H_R+sH_i+Z)\ee^{(w\lambda-\lambda)(H_R)},
 \qquad
 \zeta_j(s,Z)=(w\lambda-\rho)(sH_i+Z).
\]
Because $\operatorname{Re}\lambda$ is dominant and $H_R$ is in the positive
chamber,
$\operatorname{Re}(w\lambda-\lambda)(H_R)\leq0$.  Translation of the
polynomials by $H_R$ changes their coefficient norms by at most a fixed
power of $R$.  Part~(3) of \cref{prop:uniform-HC} therefore implies that at
least one coefficient in the displayed exponential polynomial has absolute
value at least $(X/\delta)^{-C}$.

The decomposition
$\mathfrak a=\mathbb RH_i\oplus\ker\rho$ shows that restriction to the
variables $(s,Z)$ is an isomorphism on linear functionals.  Consequently,
after equal exponentials are grouped, distinct $\zeta_j$ are separated by
$c_G\delta$.  Their real parts are bounded and their absolute values are
$O_G(Q_\lambda)$.  Applying \cref{lem:finite-exponential} gives
\begin{equation}\label{eq:leading-L2}
 \norm{\mathcal L_\lambda}_{L^2([-2,-1]\times(B_0+Z_0))}
 \gg_G X^{-C}\delta^C E_{\lambda,i,R}.
\end{equation}
Here and below the functions are evaluated at $H_R+sH_i+Z$.
Choose the exponent $A$ in \eqref{eq:HC-remainder} larger than the fixed
powers in \eqref{eq:leading-L2}.  Since
$\min_{\alpha\in\Delta}\alpha(H_R+sH_i+Z)\gg_G a_0\log X$,
choosing $a_0$ sufficiently large makes the hypothesis of
\eqref{eq:HC-remainder} valid and gives
\[
 \norm{\mathcal R_\lambda}_{L^2([-2,-1]\times(B_0+Z_0))}
 \leq\frac12
 \norm{\mathcal L_\lambda}_{L^2([-2,-1]\times(B_0+Z_0))}.
\]
This proves the desired estimate whenever $\lambda$ is
$\delta$-separated from the proper substrata.  For a stratum with no proper
substrata, take $\delta=1$; this starts the induction.

Now assume the result for every proper substratum of $S$, with a common
exponent $A_{<S}$.  The spherical function is entire in its parameter.  If
$v$ is a unit direction in $\mathfrak a_{\mathbb C}^*$, differentiation of
the spherical integral formula introduces the factor
$v(A(k\exp H))$, where $A(k\exp H)$ is the Iwasawa projection.  Since
$\norm{A(k\exp H)}\ll_G\norm H$, the integral formula and the standard
spherical-function estimate \cite[Proposition~4.6.1]{GV} give, uniformly
for $\zeta\in\mathcal U$,
\begin{equation}\label{eq:parameter-derivative}
 \abs{\partial_v\varphi_\zeta(\exp H)}
 \ll_G(1+\norm H)^C
 \ee^{-\rho(H)+\max_w\operatorname{Re}(w\zeta)(H)};
\end{equation}
see \cite[Chapter~IV, Section~4]{Helgason}.  Choose
$D>A_{<S}+C+2$ and put $\delta=X^{-D}$.  If $\lambda$ is
$\delta$-separated from every proper substratum, the direct argument proves
the claim, with a new exponent which may depend on $D$.

Otherwise choose $\lambda_0'$ on a proper substratum with
$\norm{\lambda-\lambda_0'}\ll_G\delta$.  Choose $w\in W$ so that
$\operatorname{Re}(w\lambda_0')$ is dominant and put
$\lambda_0=w\lambda_0'$.  The distance between two dominant representatives
is their quotient distance.  Since $\operatorname{Re}\lambda$ is already
dominant,
\[
 \norm{\operatorname{Re}\lambda_0-
       \operatorname{Re}\lambda}
 =d_W(\operatorname{Re}\lambda_0',\operatorname{Re}\lambda)
 \leq\norm{\operatorname{Re}\lambda_0'-
       \operatorname{Re}\lambda}
 \ll_G\delta.
\]
Moreover, $\varphi_{\lambda_0}=\varphi_{\lambda_0'}$ and
$Q_{\lambda_0}\asymp_GQ_\lambda$, so \eqref{eq:long-radius} also holds at
$\lambda_0$ after a fixed change of constants.  Along the segment joining $\lambda$ to
$\lambda_0'$, the exponent on the right of
\eqref{eq:parameter-derivative} differs from
$-\rho(H)+\operatorname{Re}\lambda(H)$ by $O_G(R\delta)$.  The mean-value
theorem therefore gives
\begin{equation}\label{eq:wall-comparison}
 \norm{\varphi_\lambda-\varphi_{\lambda_0}}_{L^2}
 \ll_G X^{-D+C}E_{\lambda,i,R},
\end{equation}
where the norm is taken over $-2\leq s\leq-1$ and $Z\in B_0+Z_0$.

The difference
\[
 H_{i,\lambda,R}(u)-H_{i,\lambda_0,R}(u)
 =(\tau_{\lambda,R}-\tau_{\lambda_0,R})(H_\circ-H_i)
\]
is independent of $u$, belongs to $\ker\rho$, and is $O_G(\delta)$.  It is
therefore absorbed by the allowed translation $Z_0$.  The induction
hypothesis at $\lambda_0$ gives
\[
 \norm{\varphi_{\lambda_0}}_{L^2}
 \gg_G X^{-A_{<S}}E_{\lambda,i,R};
\]
the exponent at $\lambda_0$ is comparable with the one displayed because
their logarithms differ by $O_G(R\delta)$.  Our choice of $D$ makes
\eqref{eq:wall-comparison} at most half this lower bound.  This completes the
induction.  Squaring the resulting estimate proves
\eqref{eq:tube-L2-lower}.  Finally, since only finitely many strata occur,
$a_0$ may be chosen once so that the remainder estimate used above holds in
every direct step.  We then enlarge $A_0$ so that
\eqref{eq:long-radius} implies $R-3\geq2\tau_{\lambda,R}$.
\end{proof}

\begin{remark}
The polynomial $P_{\lambda,e}$ in part~(3) of
\cref{prop:uniform-HC} is essential.  Its quantitative nonvanishing supplies
the coefficient to which \cref{lem:finite-exponential} is applied; the size
of an exponential alone would not give a lower bound for their sum.
\end{remark}

\subsection{The mass estimate}

For a unitary parameter $\lambda$, write $p_\lambda=p(\pi_\lambda)$.

\begin{proposition}\label{prop:critical-mass}
There are $B,C>0$, depending only on $G$, with the following property.  Let
$\vartheta=\vartheta(\lambda)>0$ and suppose
\begin{equation}\label{eq:critical-hypotheses}
 R\vartheta\geq C\log(2+Q_\lambda R).
\end{equation}
Then a critical index $i$ satisfies
\begin{equation}\label{eq:critical-mass}
 I_{\lambda,i,R}:=
 \int_G\chi_{\lambda,i,R}(g)\ee^{\ell(g)/p_\lambda}
       \abs{\varphi_\lambda(g)}^2\dd g
 \gg_GQ_\lambda^{-C}R^{-C}
       \ee^{R(1-1/p_\lambda)}.
\end{equation}
If $\nu$ is a spherical unitary parameter and
\begin{equation}\label{eq:nearby-parameter}
 d_W(\nu,\lambda)\leq Q_\lambda^{-B}R^{-B},
\end{equation}
then, for
\[
 \beta_{\lambda,i,R}(g)
 =\chi_{\lambda,i,R}(g)\ee^{\ell(g)/p_\lambda}
       \overline{\varphi_\lambda(g^{-1})},
\]
one has
\begin{equation}\label{eq:nearby-transform}
 \abs{\widehat\beta_{\lambda,i,R}(\nu)}
 \gg_GQ_\lambda^{-C}R^{-C}
       \ee^{R(1-1/p_\nu)}.
\end{equation}
Finally, on the support of $\chi_{\lambda,i,R}$,
\begin{equation}\label{eq:endpoint-upper}
 \ee^{\ell(g)/p_\lambda}\abs{\varphi_\lambda(g)}
 \ll_GR^C Q_\lambda^C.
\end{equation}
\end{proposition}

\begin{proof}
We first need an upper bound.  The standard uniform estimate for the
elementary spherical function is
\begin{equation}\label{eq:spherical-endpoint}
 \abs{\varphi_\lambda(\exp H)}
 \ll_G(1+\norm H)^C
 \ee^{-\rho(H)+\max_w\operatorname{Re}(w\lambda)(H)}
\end{equation}
see \cite[Proposition~4.6.1]{GV}.  By \cref{lem:p-poly},
\[
 \max_w\operatorname{Re}(w\lambda)(H)
 \leq\vartheta(\lambda)\rho(H)
 \qquad(H\in\overline{\mathfrak a^+}).
\]
Since $1-\vartheta(\lambda)=2/p_\lambda$ and
$\ell(\exp H)=2\rho(H)$, the exponential in
\eqref{eq:spherical-endpoint} is at most
$\ee^{-\ell(\exp H)/p_\lambda}$.  On the support of the cutoff, the
remaining polynomial is $O_G(Q_\lambda^CR^C)$.  This proves
\eqref{eq:endpoint-upper}.

The constant in \eqref{eq:critical-hypotheses} is chosen large enough that
\eqref{eq:long-radius} holds.  For a critical index $i$,
\cref{lem:wall-induction} and \eqref{eq:critical-loss} give
\[
 \int_{R-2}^{R-1}\int_{B_0}
 \abs{\varphi_\lambda
  (\exp(H_{i,\lambda,R}(u)+Z))}^2\dd Z\dd u
 \gg_GQ_\lambda^{-C}R^{-C}\ee^{-(1-\vartheta)R}.
\]
On the region of integration, \eqref{eq:jacobian-tube} gives
$J(H)\asymp_G\ee^R$ and
$\ee^{\ell(g)/p_\lambda}\asymp_G\ee^{R/p_\lambda}$.  Hence
\[
 I_{\lambda,i,R}
 \gg_GQ_\lambda^{-C}R^{-C}
 \ee^{-(1-\vartheta)R+R+R/p_\lambda}
 =Q_\lambda^{-C}R^{-C}\ee^{R(1-1/p_\lambda)}.
\]

The definition of the spherical transform gives, at $\nu=\lambda$,
\[
 \widehat\beta_{\lambda,i,R}(\lambda)=I_{\lambda,i,R}.
\]
Indeed,
$\overline{\varphi_\lambda(g^{-1})}\varphi_\lambda(g^{-1})
=\abs{\varphi_\lambda(g)}^2$ for a unitary spherical parameter.
Choose the Weyl representative of $\nu$ for which
$\norm{\nu-\lambda}=d_W(\nu,\lambda)$, and let $\zeta$ lie on the segment
from $\lambda$ to $\nu$.  Since \eqref{eq:theta-poly} is the maximum of
finitely many linear functions,
\[
 \abs{\vartheta(\zeta)-\vartheta(\lambda)}
 \ll_G\norm{\zeta-\lambda}.
\]
For $B\geq1$, condition \eqref{eq:nearby-parameter} therefore implies, on
the support of $\beta_{\lambda,i,R}$,
\[
 -\rho(H)+\max_w\operatorname{Re}(w\zeta)(H)
 \leq-\frac{\ell(\exp H)}{p_\lambda}+O_G(1).
\]
Differentiating the spherical integral formula as in
\cite[Chapter~IV, Section~4]{Helgason} introduces a factor of size
$O_G(1+\norm H)$; the absolute value of the resulting integral is therefore
at most $O_G(1+\norm H)\varphi_{\operatorname{Re}\zeta}(\exp H)$.
Hence \eqref{eq:spherical-endpoint} and
\eqref{eq:endpoint-upper} give
\[
 \abs{\beta_{\lambda,i,R}(g)}
 \norm{\nabla_\zeta\varphi_\zeta(g^{-1})}
 \ll_GQ_\lambda^CR^C\ee^{-\ell(g)/p_\lambda}
\]
on this support.  Cartan integration gives
$\vol\{g:\ell(g)\leq R\}\ll_GR^C\ee^R$.  Integration now yields
\begin{equation}\label{eq:transform-derivative}
 \norm{\nabla_\zeta\widehat\beta_{\lambda,i,R}(\zeta)}
 \ll_GQ_\lambda^CR^C\ee^{R(1-1/p_\lambda)}
\end{equation}
along the segment; here $\nabla_\zeta$ is the gradient for the fixed inner
product on the parameter space.  Choose $B$ larger than the sum of the fixed
powers in \eqref{eq:critical-mass} and
\eqref{eq:transform-derivative}.  The mean-value theorem and
\eqref{eq:nearby-parameter} then give
\[
 \abs{\widehat\beta_{\lambda,i,R}(\nu)
       -\widehat\beta_{\lambda,i,R}(\lambda)}
 \leq\frac12 I_{\lambda,i,R}.
\]
Finally, the Lipschitz estimate for $\vartheta$ and \cref{lem:p-poly} give
\[
 \abs{p_\nu^{-1}-p_\lambda^{-1}}
 \ll_G d_W(\nu,\lambda),
 \qquad
 R\abs{p_\nu^{-1}-p_\lambda^{-1}}=O_G(1).
\]
Thus replacing $p_\lambda$ by $p_\nu$ changes the final exponential by only
a bounded factor.
\end{proof}

\section{Positive kernels}

We now use the mass estimate to build a positive test function for the
pre-trace formula.  This is the construction of
\cite[Section~2.3]{OhDensity}; compare also the truncated kernel in
\cite{LutskoDensity}.  The preceding section allows us to place the test
function in a critical Cartan direction without losing uniformity in the
spectral parameter.

For the index $i$ supplied by \cref{prop:critical-mass}, define
\[
       \Phi_{\lambda,i,R}=\beta_{\lambda,i,R}^* *
                            \beta_{\lambda,i,R}.
\]

\begin{lemma}
Under the hypotheses of \cref{prop:critical-mass},
\begin{align}
 \supp\Phi_{\lambda,i,R}
 &\subset\{g:\ell(g)\leq2R+C\},                    \label{eq:Phi-support}\\
 \abs{\Phi_{\lambda,i,R}(g)}
 &\ll_GQ_\lambda^CR^C
       \ee^{R-\ell(g)/2}
       \mathbf1_{\ell(g)\leq2R+C},                 \label{eq:Phi-upper}\\
 \widehat\Phi_{\lambda,i,R}(\nu)
 &\gg_GQ_\lambda^{-C}R^{-C}
       \ee^{2R(1-1/p_\nu)}                         \label{eq:Phi-lower}
\end{align}
whenever \eqref{eq:nearby-parameter} holds.  Moreover,
$\pi(\Phi_{\lambda,i,R})\geq0$ for every unitary representation $\pi$.
\end{lemma}

\begin{proof}
By the definition of convolution,
\[
 \Phi_{\lambda,i,R}(g)
 =\int_G\overline{\beta_{\lambda,i,R}(h^{-1})}
            \beta_{\lambda,i,R}(h^{-1}g)\dd h.
\]
If the integrand is nonzero, then
$h^{-1}$ and $h^{-1}g$ lie in the support of $\beta_{\lambda,i,R}$.
Consequently, \eqref{eq:length-properties} gives
\[
 \ell(g)\leq\ell(h)+\ell(h^{-1}g)
            =\ell(h^{-1})+\ell(h^{-1}g)\leq2R+C,
\]
which proves \eqref{eq:Phi-support}.  By
\eqref{eq:endpoint-upper},
\[
 \abs{\beta_{\lambda,i,R}(g)}
 \ll Q_\lambda^CR^C\mathbf1_{\ell(g)\leq R+C}.
\]
The Cartan convolution estimate
\[
 \mathbf1_{\ell\leq R+C}*\mathbf1_{\ell\leq R+C}(g)
 \ll_G R^C\ee^{R-\ell(g)/2}
       \mathbf1_{\ell(g)\leq2R+C}
\]
is \cite[Lemma 4.18]{GK}, and gives \eqref{eq:Phi-upper}.  On the spherical
dual,
\[
 \widehat\Phi_{\lambda,i,R}(\nu)
 =\abs{\widehat\beta_{\lambda,i,R}(\nu)}^2,
\]
so \eqref{eq:Phi-lower} follows from
\eqref{eq:nearby-transform}.  For an arbitrary unitary representation,
$\pi(\beta^*)=\pi(\beta)^*$ and hence
\[
 \pi(\Phi_{\lambda,i,R})
 =\pi(\beta_{\lambda,i,R})^*\pi(\beta_{\lambda,i,R}).
\]
This is a positive semidefinite operator, which is what
$\pi(\Phi_{\lambda,i,R})\geq0$ means.
\end{proof}

We now cover a spectral family.  For $L\geq0$ and $\eta>0$, let
\[
 \cA_{L,\eta}=\{\lambda:\pi_\lambda\in\widehat G_{\mathrm{sph}},
       \ \Lambda(\pi_\lambda)\leq L,
       \ \vartheta(\lambda)\geq\eta\}.
\]
This is a compact subset of the finite-dimensional parameter space
$\mathfrak a_{\mathbb C}^*/W$, and $Q_\lambda\ll_G(1+L)^{1/2}$ on it by
\eqref{eq:height-comparison}.

\begin{proposition}\label{prop:positive-family}
Suppose
\begin{equation}\label{eq:family-radius}
 R\eta\geq C\log((2+L)(2+R)).
\end{equation}
There is
$F_{R,L,\eta}=F_{R,L,\eta}^*\in C_c^\infty(K\backslash G/K)$ such that
$\pi(F_{R,L,\eta})\geq0$ for every unitary $\pi$ and
\begin{align}
 \abs{F_{R,L,\eta}(g)}
 &\ll_G(1+L)^CR^C
       \ee^{R-\ell(g)/2}
       \mathbf1_{\ell(g)\leq2R+C},                 \label{eq:family-upper}\\
 \widehat F_{R,L,\eta}(\nu)
 &\gg_G(1+L)^{-C}R^{-C}
       \ee^{2R(1-1/p_\nu)}
       \qquad(\nu\in\cA_{L,\eta}).               \label{eq:family-lower}
\end{align}
\end{proposition}

\begin{proof}
Set $r_0=(1+L)^{-B}R^{-B}$.  An $r_0$-net is a finite set
$\{\lambda_j\}\subset\cA_{L,\eta}$ such that every
$\nu\in\cA_{L,\eta}$ satisfies $d_W(\nu,\lambda_j)\leq r_0$ for some $j$.
Since the parameter space has fixed dimension and $\cA_{L,\eta}$ has
diameter $O_G((1+L)^{1/2})$, a volume comparison gives a net whose
cardinality is at most a fixed power of $(1+L)R$.  For each net point,
\eqref{eq:family-radius} implies the hypothesis of
\cref{prop:critical-mass}; choose a critical index $i_j$ and set
\[
                  F_{R,L,\eta}=\sum_j
                     \Phi_{\lambda_j,i_j,R}.
\]
Every summand is positive.  Given $\nu\in\cA_{L,\eta}$, choose a net point
within distance $r_0$.  With $B$ large enough, this is the neighborhood in
\eqref{eq:nearby-parameter}, and \eqref{eq:Phi-lower} gives
\eqref{eq:family-lower}.  Summing \eqref{eq:Phi-upper} gives
\eqref{eq:family-upper}.  The number of net points and all spectral-height
losses are absorbed by the fixed powers of $(1+L)R$.
\end{proof}

\section{Proof of the density theorem}\label{sec:density-proof}

\begin{proof}[Proof of \cref{thm:main}]
Fix $\epsilon>0$ and put
\[
                              \theta=1-\frac2p.
\]
This is the threshold for $\vartheta(\lambda)$ corresponding to the
condition $p(\pi_\lambda)\geq p$.  If
$\theta\leq\epsilon/4$, then $\alpha\theta\leq\epsilon/4$, and the local
Weyl estimate \eqref{eq:local-weyl} implies \eqref{eq:main}, since
$V_N\leq V_N^{1-\alpha\theta+\epsilon}$.  We may therefore suppose
$\theta>\epsilon/4$.  The trivial representation has multiplicity one; we
remove it from the count and absorb this contribution into the final implied
constant.

Choose
\[
                       R=\alpha\log V_N-C_1,
\]
where $C_1$ is sufficiently large for \cref{lem:geometric-trace}.  When
$R\leq1$, the volume $V_N$ is bounded in terms of $\alpha$ and $C_1$, and
the assertion follows from \eqref{eq:local-weyl} after changing the implied
constant.  We may therefore assume $R>1$.  First suppose
\begin{equation}\label{eq:R-long}
                 R\theta\geq C\log((2+L)(2+R)).
\end{equation}
Use \cref{prop:positive-family} with $\eta=\theta$.  By
\cref{lem:p-poly}, every representation counted by
$M_{\mathrm{sph}}(L,\Gamma_N,p)$ has
$\vartheta(\lambda)\geq\theta$ and therefore belongs to
$\mathcal A_{L,\theta}$.  The spectral trace formula
\eqref{eq:spectral-trace}, positivity, and
\eqref{eq:family-lower} give
\[
 \Tr R_{\Gamma_N}(F_{R,L,\theta})
 \gg_G(1+L)^{-C}R^{-C}
       \ee^{2R(1-1/p)}M_{\mathrm{sph}}(L,\Gamma_N,p).
\]
Indeed, $\vartheta(\nu)\geq\theta$ implies $p_\nu\geq p$, and hence
$1-1/p_\nu\geq1-1/p$.

On the geometric side, \cref{lem:geometric-trace} and
\eqref{eq:family-upper} give
\[
 \Tr R_{\Gamma_N}(F_{R,L,\theta})
 \ll_G(1+L)^CR^C V_N^{1+\epsilon}\ee^R.
\]
Divide the geometric upper bound by the contribution of one representation
in the spectral lower bound.  Absorbing the fixed powers of $R$ into
$V_N^\epsilon$ gives
\begin{align*}
 M_{\mathrm{sph}}(L,\Gamma_N,p)
 &\ll_G(1+L)^C
 V_N^{1+\epsilon}\ee^{-R(1-2/p)}\\
 &\ll_{G,\Gamma_1,\alpha,\epsilon}(1+L)^C
 V_N^{1-\alpha(1-2/p)+2\epsilon}.
\end{align*}
Replacing $\epsilon$ by $\epsilon/2$ proves \eqref{eq:main} in the long
range.

It remains to consider the complementary, short range, in which
\eqref{eq:R-long} fails.  Then
\begin{equation}\label{eq:short-range}
 R\theta\ll_G\log(2+L)+\log(2+R).
\end{equation}
Since $R=\alpha\log V_N+O_{G,\Gamma_1}(1)$, for all sufficiently large
$V_N$ we have
\[
 C_G\log(2+R)+O_{G,\Gamma_1}(1)
 \leq\frac\epsilon2\log V_N.
\]
It follows from \eqref{eq:short-range} that
\[
 V_N^{\alpha\theta-\epsilon/2}
 \ll_{G,\Gamma_1,\alpha,\epsilon}(1+L)^C
\]
with a fixed $C=C(G)$.  The remaining bounded values of $V_N$ are absorbed
by the implied constant.  Therefore \eqref{eq:local-weyl} gives
\[
 M_{\mathrm{sph}}(L,\Gamma_N,p)
 \ll_{G,\Gamma_1}V_N(1+L)^d
 \ll_{G,\Gamma_1,\alpha,\epsilon}
 (1+L)^CV_N^{1-\alpha\theta+\epsilon}.
\]
This completes the proof.
\end{proof}

\begin{proof}[Proof of \cref{cor:equivalence}]
The forward implication is \cref{thm:main}.  The reverse implication is
\cite[Theorem 1.6 and Section 7.2]{GK}; the proof preserves the parameter
$\alpha$.  The two formulations of weak injective radius agree by
\cref{rem:wirp-formulations}.
\end{proof}

\medskip
\noindent\hrulefill

\subsection*{Acknowledgements}

The author thanks Tobias Weich for helpful comments on an earlier version of
this paper.

\subsection*{Acknowledgement of AI assistance}

OpenAI's ChatGPT was used during the development of this note to explore the
kernel construction, check calculations, locate references, and assist with
drafting and revision.  The author assumes full responsibility for the
contents of the paper.

\end{document}